\documentclass[article]{amsart}
\usepackage{amssymb,amsfonts,amsmath,amsthm}
\usepackage[all,arc]{xy}
\usepackage{enumerate}
\usepackage{mathrsfs}
\usepackage[toc,page]{appendix}
\usepackage{graphicx}
\usepackage{tabularx}
\usepackage{multirow} 
\usepackage{caption} 
\usepackage{url}
\usepackage{color}
\usepackage{tikz-cd}  
\usepackage{mathdots} 
\usepackage{romannum} 
\usepackage{hyperref} 
\usepackage{float} 

\newtheorem{thm}{Theorem}[section]

\newtheorem*{thm*}{Theorem}
\newtheorem*{cor*}{Corollary}
\newtheorem*{prop*}{Proposition}
\newtheorem{cor}[thm]{Corollary}
\newtheorem{prop}[thm]{Proposition}
\newtheorem{lem}[thm]{Lemma}
\newtheorem{conj}[thm]{Conjecture}

\theoremstyle{definition}
\newtheorem{defn}[thm]{Definition}

\newtheorem{conv*}{Convention}
\newtheorem{exmp}[thm]{Example}

\newtheorem*{notn*}{Notation}

\theoremstyle{remark}
\newtheorem{rem}[thm]{Remark}

\newtheorem*{idea*}{Idea}

\newcommand{\Spec}{{\rm Spec}}

\makeatletter
\let\c@equation\c@thm
\makeatother
\numberwithin{thm}{section}
\numberwithin{equation}{section}

\title[Cyclic Spectral Data]{Cyclic Spectral Data}

\author{Hao Sun}

\begin{document}
\pagenumbering{arabic}
\maketitle
\begin{abstract}
In this paper, we study cyclic spectral data on smooth varieties over an algebraically closed field in characteristic zero. We show that the Hitchin fiber of any nonzero cyclic spectral data is non-empty. We also construct sections of the Hitchin morphism for cyclic spectral data and tautological (flat) families of cyclic spectral covers.
\end{abstract}
	
\flushbottom
	
	
\renewcommand{\thefootnote}{\fnsymbol{footnote}}
	
\section{Introduction}

We fix two positive integers $d$ for the dimension and $n$ for the rank. Let $X$ be a connected smooth variety over an algebraically closed field $k$ in characteristic zero with cotangent sheaf $\Omega_X^1$. Let $T^*_X:={\rm tot}(\Omega_X^1)$ be the total space. Define
\begin{align*}
    \mathscr{A}_X:=\bigoplus_{i=1}^n H^0(X, S^i \Omega_X^1),
\end{align*}
which is called the \emph{Hitchin base}. Let $\mathscr{M}_X$ be the moduli stack of Higgs bundles of rank $n$ on $X$. The Hitchin morphism
$h_X: \mathscr{M}_X \rightarrow \mathscr{A}_X$ sends a Higgs bundle $(E,\theta)$ to the coefficients of the characteristic polynomial of $\theta$ \cite[\S 6, Hitchin's proper map]{Sim94b}. Given any $\underline{a} = (a_1,\dots,a_n) \in \mathscr{A}_X$, we can construct a closed subscheme $X_a$ of $T^*_X$, which is defined by the equation $t^n + a_1 t^{n-1} + \dots + a_n = 0$. 
\begin{center}
\begin{tikzcd}
X_a \arrow[r, hook] \arrow[rd, "\pi_a" description] & T^*_X \arrow[d] \\
 & X
\end{tikzcd}
\end{center}
In the case of algebraic curves, $h_X$ is surjective, and when $X_a$ is integral, the category of torsion‑free sheaves $\mathcal{L}$ of rank one on $X_a$ and the category of Higgs bundles $(E,\theta)$ satisfying $h_X((E,\theta)) = a$ \cite[Proposition 3.6]{BNR89}. This correspondence provides a new approach to studying the fibers of the Hitchin morphism and the moduli space of Higgs bundles. However, in higher dimensions, the situation becomes significantly more complicated. Chen and Ng\^o showed that the Hitchin morphism factors through a closed subscheme $\mathscr{B}_X$ of $\mathscr{A}_X$ 
\begin{center}
\begin{tikzcd}
& & \mathscr{M}_X \arrow[dd, "h_X" description] \arrow[lldd, "{\rm sd}_X"  description, dotted] \\
& & \\
\mathscr{B}_X \arrow[rr, hook, "\iota_X" description] & & \mathscr{A}_X
\end{tikzcd}
\end{center}
defined as
\begin{align*}
    \mathscr{B}_X:={\rm Sect}(X, {\rm Chow}_n(T^*_X/X)),
\end{align*}
where the relative Chow scheme ${\rm Chow}_n(T^*_X/X)$ is defined as 
\begin{align*}
    {\rm Chow}_n(T^*_X/X):=T^*_X \times_X \dots \times_X T^*_X / \mathfrak{S}^n,
\end{align*}
and $\mathfrak{S}^n$ is the symmetric group. The morphism ${\rm sd}_X$ is called the \emph{spectral data morphism}. 

Given a $k$-point $b \in \mathscr{B}_X(k) \subseteq \mathscr{A}_X(k)$, let $X_b$ be the corresponding spectral cover. Consider an open subset $\mathscr{B}_X^{\heartsuit} \subseteq \mathscr{B}_X$ consisting of sections $b: X \rightarrow {\rm Chow}_n(T^*_X/X)$ that map the generic point of $X$ to multiplicity free $0$-cycles. Chen and Ng\^o proved the following correspondences on this open locus $\mathscr{B}_X^{\heartsuit}$:
\begin{prop}[Proposition 3.2 in \cite{CN18}  and Proposition 6.3 in \cite{CN20}]
Given any $b \in \mathscr{B}^{\heartsuit}_X(k)$, the fiber ${\rm sd}_X^{-1}(b)$ is isomorphic to the algebraic stack of maximal Cohen-Macaulay sheaves of generic rank one on the spectral cover $X_b$.
\end{prop}
Since the spectral cover $X_b$ is not well-behaved (e.g., non-reduced, non-Cohen-Macaulay), it is not clear how to prove the existence of maximal Cohen-Macaulay sheaves of generic rank one on $X_b$. Therefore, they proposed the following conjecture:
\begin{conj}[Conjecture 5.2 in \cite{CN20}]\label{conj_CN5.2}
Let $X$ be a smooth proper variety. For every $b \in \mathscr{B}_X(k)$, the fiber ${\rm sd}_X^{-1}(b)$ is non-empty.
\end{conj}
In the case of smooth proper algebraic surfaces, Chen and Ng\^o constructed a Cohen-Macaulay spectral cover $X^{\rm CM}_b$ and proved that the Hitchin fiber $h^{-1}_X(b)$ for $b \in \mathscr{B}^{\heartsuit}_X(k)$ is isomorphic to the stack of Cohen-Macaulay sheaves of generic rank one on $X^{\rm CM}_b$. Since $X^{\rm CM}_b$ is Cohen-Macaulay, the fiber ${\rm sd}_X^{-1}(b)$ is automatically  non-empty \cite[Theorem 7.3]{CN20}. Moreover, this conjecture has been confirmed by Song and the author in a subsequent work \cite{SS24}. He and Liu studied the rank two case in detail and confirmed this conjecture \cite{HL24}.

In this paper, we consider a special case of the spectral data, which are called the \emph{cyclic spectral data}.

\begin{defn}\label{defn_cyc_spec_data}
A geometric point $b \in \mathscr{B}_X(k)$ is called \emph{cyclic} if 
\begin{align*}
    \iota_X(b) = (0,\dots,0,-a) \in \mathscr{A}_X(k)
\end{align*}
for some $a \in H^0(X,S^n \Omega_X^1)$. Such geometric points $b \in \mathscr{B}_X(k)$ and the corresponding elements $a \in H^0(X,S^n\Omega_X^1)$ are called \emph{cyclic spectral data}. 

The subscheme $\mathscr{B}_X^{\rm cyc} \subseteq \mathscr{B}_X$ of cyclic spectral data is defined as
\begin{align*}
    \mathscr{B}^{\rm cyc}_X := \mathscr{B}_X \times_{\mathscr{A}_X} H^0(X,S^n \Omega_X^1)
\end{align*}
Moreover, let $\mathscr{B}^{\rm cyc, \times}_X := \mathscr{B}_X^{\rm cyc} \backslash \{0\}$ be the subscheme of nonzero cyclic spectral data.
\end{defn}

The spectral cover given by cyclic spectral data is called a \emph{cyclic spectral cover}. In the case of curves, there are many interesting works from different viewpoints \cite{Bar15,BGG03,DL20}. 

Since it is more convenient to work on elements in $H^0(X,S^n \Omega_X^1)$, we prefer to use the corresponding element $a \in H^0(X,S^n \Omega_X^1)$ of the given spectral data $b$ to construct the spectral cover, and we use the notations $X_a$ for the spectral cover and $h_X^{-1}(a)$ for the Hitchin fiber.

We first prove the non-emptiness of the Hitchin fiber of cyclic spectral data via a discussion of local pictures in \S\ref{sect_nonempty}.
\begin{thm}[Theorem \ref{thm_nor_MCM}]\label{thm_nonempty}
Let $X$ be a connected smooth variety over $k$. Let $a$ be nonzero cyclic spectral data. The fiber $h_X^{-1}(a)$ is non-empty.
\end{thm}

We want to emphasize that the theorem does not require $X$ to be proper. Moreover, although Conjecture \ref{conj_CN5.2} is stated for proper smooth varieties, it is false for smooth non-proper varieties, which was pointed out to us by Jie Liu and Jinxing Xu in private communications. 

In \S\ref{sect_Hit_sect}, we provide a second proof of Theorem \ref{thm_nonempty} by generalizing He--Liu's approach \cite[\S3 and \S4]{HL24} to higher rank. As a by-product, we construct a Cohen-Macaulayfication $X'_a$ of $X_a$ and prove an analogue of the Beauville--Narasimhan--Ramanan correspondence.
\begin{thm}[Theorem \ref{thm_BNR_corr}]
Let $a$ be nonzero cyclic spectral data. The Hitchin fiber $h^{-1}_X(a)$ is isomorphic to the stack $\mathcal{MCM}_1(X'_a)$ of maximal Cohen-Macaulay sheaves of generic rank one on $X'_a$.
\end{thm}

This theorem recovers Chen--Ng\^o's result in the rank-two cyclic case  \cite[Theorem 7.3]{CN20} and $X'_a$ is exactly the Cohen-Macaulay spectral surface $X^{\rm CM}_a$ constructed in \cite[Proposition 7.2]{CN20}. When the rank of Higgs bundles is two, this result recovers He--Liu's result \cite[Theorem 4.6]{HL24}. 

The next result in this paper is to construct analogues of the Hitchin section \cite{Hit92}. As a generalization of Bogomolov--De Oliveira \cite[Theorem 2.3]{BdO11} and He--Liu's result \cite[Proposition 3.6]{HL24}, any nonzero cyclic spectral data $a$ determines a triple $(\mathscr{L},\tau,\alpha)$, where $\mathscr{L}$ is a line bundle, $\tau \in H^0(X,\mathscr{L}^n)$ and $\alpha \in H^0(X,\Omega_X^1 \otimes \mathscr{L}^{-1})$, such that $a=\tau \alpha^n$ and $\alpha$ does not vanish in codimension one (Proposition \ref{prop_sym_diff} and Corollary \ref{cor_decomp}). Moreover, this triple is unique up to equivalence. Based on the corresponding triple $(\mathscr{L},\tau,\alpha)$ of $a$, we define a set-theoretic map
\begin{align*}
    \mathscr{B}^{\rm cyc,\times}_X(k) \rightarrow \mathscr{M}^{\rm cyc}_X(k), \quad a \mapsto (E_a,\theta_a),
\end{align*}
where $\mathscr{M}^{\rm cyc,\times}_X := \mathscr{M}_X \times_{\mathscr{B}_X} \mathscr{B}^{\rm cyc,\times}_X$ is the moduli stack of cyclic Higgs bundles,
\begin{align*}
    E_a: = \mathcal{O}_X \oplus \mathscr{L}^{-1} \oplus \dots \oplus \mathscr{L}^{-(n-1)}.
\end{align*}
and
\begin{align*}
\theta_a = 
\begin{pmatrix}
     & \alpha & & \\
    &  & \ddots & \\
    & & & \alpha \\
    \tau \alpha & & & \\
\end{pmatrix}.
\end{align*}
We try to make this set-theoretic map into algebraic morphisms. We define
\begin{align*}
    U_{\mathscr{L}}: = H^0(X,\mathscr{L}^n), \, W_{\mathscr{L}} := H^0(X, \Omega_X^1 \otimes \mathscr{L}^{-1}),
\end{align*}
and let $W^{\rm sat}_\mathscr{L} \subseteq W_\mathscr{L}$ be the subset that consists of sections $\alpha$ that do not vanish in codimension one. We introduce
\begin{align*}
   V_{\mathscr{L}}^{\rm sat,\times} :=  (U_{\mathscr{L}} \backslash \{0\} ) \times W^{\rm sat}_\mathscr{L}
\end{align*}
with a $\mathbb{G}_m$-action defined as
\begin{align*}
    c \cdot (\tau,\alpha) : = (c^{-n} \tau, c \alpha),
\end{align*}
With respect to the above notations, we define a set of line bundles
\begin{align*}
    \mathcal{P}_X := \{ \mathscr{L} \in {\rm Pic}(X) \, | \,  U_\mathscr{L} \neq 0, \, W^{\rm sat}_\mathscr{L} \neq 0  \}.
\end{align*}

\begin{thm}[Theorem \ref{thm_stack_Hit_sect}]
For each $\mathscr{L} \in \mathcal{P}_X$, there exists a morphism $\bar{s}_{\mathscr{L}}: [V^{\rm sat,\times}_{\mathscr{L}} / \mathbb{G}_m] \rightarrow \mathscr{M}^{\rm cyc,\times}_X$ such that $\bar{q}_{\mathscr{L}} = {\rm sd}_X \circ \bar{s}_{\mathscr{L}}$, i.e.,
\begin{center}
\begin{tikzcd}
{[V_{\mathscr{L}}^{\mathrm{sat}, \times} / \mathbb{G}_m]} \arrow[r,"\bar{s}_\mathscr{L}", dotted] \arrow[rd,"\bar{q}_\mathscr{L}"] & \mathscr{M}^{\rm cyc,\times}_X \arrow[d, "{\rm sd}_X"] \\
& \mathscr{B}^{\rm cyc, \times}_X \, .
\end{tikzcd}
\end{center}
\end{thm}

In the projective case, $[V^{\rm sat,\times}_{\mathscr{L}} / \mathbb{G}_m]$ is represented by a reduced scheme $\mathscr{B}^{\rm cyc}_\mathscr{L}$ (Lemma \ref{lem_B_cyc_L}), which is a locally closed subscheme of $\mathscr{B}^{\rm cyc,\times}_X$, and we have the following result as a direct application of the above theorem.

\begin{thm}[Theorem \ref{thm_proj_Hit_sect}]
For each $\mathscr{L} \in \mathcal{P}_X$, there exists a morphism $s_{\mathscr{L}} : \mathscr{B}^{\rm cyc}_{\mathscr{L}} \rightarrow \mathscr{M}^{\rm cyc,\times}_X$ such that ${\rm sd}_X \circ s_{\mathscr{L}} = \iota_\mathscr{L}$.
\begin{center}
\begin{tikzcd}
{[V_{\mathscr{L}}^{\mathrm{sat}, \times} / \mathbb{G}_m]} \arrow[r,"\bar{s}_\mathscr{L}"] \arrow[d,"\cong"] & \mathscr{M}^{\rm cyc,\times}_X \arrow[d, "{\rm sd}_X"] \\
\mathscr{B}^{\rm cyc}_{\mathscr{L}} \arrow[r, hook, "\iota_\mathscr{L}"] \arrow[ru, "s_\mathscr{L}" description, dotted] & \mathscr{B}^{\rm cyc, \times}_X \, .
\end{tikzcd}
\end{center}
\end{thm}

We do not know whether the set-theoretic map
\begin{align*}
    \mathscr{B}^{\rm cyc,\times}_X(k) \rightarrow \mathscr{M}_X(k), \quad a \mapsto (E_a,\theta_a).
\end{align*}
comes from an algebraic morphism or not. We give an example in Remark \ref{rem_non-alg} discussing the difficulty.

\vspace{2mm}
{\bf Acknowledgement.} The author would like to thank Jie Liu and Jinxing Xu for helpful discussions. This project is funded by the National Key R\&D Program of China (No. 2022YFA1006600). 

\vspace{2mm}

{\bf Disclosure of AI.} The AI tools (ChatGPT, gpt-5.6-sol and Deepseek, deepseek-V4-Pro) were used to provide examples, proofread the paper and improve the exposition. The authors are responsible for all assertions in this paper.

\section{Non-emptiness of the Hitchin Fiber of Cyclic Spectral Data}\label{sect_nonempty}

Recall that $X$ is a connected smooth variety of dimension $d$. Let $x \in X$ be a point and $R:=\mathcal{O}_{X,x}$ a regular local ring with maximal ideal $\mathfrak{m}$. Now we will work on this regular local ring $(R, \mathfrak{m})$ and use the notations $\Omega_R^1:=\Omega^1_{\Spec \, R}$ and $T^*_R := T^*_{\Spec \, R}$. Note that $\Omega^1_R$ is a free $R$-module, and we then fix a basis $\{\omega_1,\dots,\omega_d\}$ with fiber coordinates $\{t_1,\dots,t_d\}$. Therefore, $T^*_R \cong \Spec \, R[t_1,\dots,t_d]$. 

To simplify notation, we define $\underline{t}:=(t_1,\dots,t_d)$ so that $T^*_{R} \cong R[\underline{t}]$. Let $\alpha = (\alpha_1,\dots,\alpha_d)$ be a $d$-tuple of non-negative integers. We define $|\alpha|:= \alpha_1 + \dots + \alpha_d$, and the sum of two $d$-tuples $\alpha$ and $\beta$ is defined as 
\begin{align*}
    \alpha + \beta = (\alpha_1 + \beta_1, \dots , \alpha_d + \beta_d).
\end{align*}
With respect to the above notation, we define
\begin{align*}
    \underline{t}^\alpha := t_1^{\alpha_1} \dots t_d^{\alpha_d}.
\end{align*}

Let $a \in H^0(X,S^n \Omega_X^1)$ be a section. We obtain a section $a_R \in H^0(\Spec \, R, S^n \Omega_R^1)$ by pulling back $a$ via the canonical morphism $\Spec \, R \rightarrow X$. Since we will focus on the case of regular local rings, if there is no ambiguity, we still use the notation $a:=a_R$ for simplicity. Then, we write $a$ as
\begin{align*}
    a = \sum_{|\alpha| = n} {n \choose \alpha} a_{\alpha} \underline{\omega}^{\alpha},
\end{align*}
where ${n \choose \alpha} := {n \choose \alpha_1 \, \dots \, \alpha_d}$ and $a_{\alpha} \in R$ for each $\alpha$. 

\begin{lem}\label{lem_cond_*}
If $a = \sum_{|\alpha| = n} {n \choose \alpha} a_{\alpha} \underline{\omega}^{\alpha}$ is cyclic spectral data, it satisfies
\begin{equation}\tag{$\ast$}\label{eq_cond}
    a_{\alpha} a_{\beta} = a_{\gamma} a_{\delta},
\end{equation}
whenever $\alpha+\beta = \gamma +\delta$.
\end{lem}

\begin{proof}
Let $\bar{K}$ be the algebraic closure of the fraction field $K={\rm Frac} \, R$. Since $a = \sum_{|\alpha| = n} {n \choose \alpha} a_{\alpha} \underline{\omega}^{\alpha}$ is spectral data, we obtain a point $b_{\bar{K}} \in {\rm Chow}_n(T^*_{\bar{K}} / \bar{K})$ by base change to $\bar{K}$. Since $\bar{K}$ is algebraically closed, the point $b_{\bar{K}}$ can be represented by an unordered $n$-tuple $[v_1,\dots,v_n]$, where $v_i \in T^*_{\bar{K}}$. Note that $\prod\limits_{i=1}^n (t-v_i) = t^n -a$. Therefore, there exist $w_1, \dots,w_d \in \bar{K}$ such that $a = (w_1 \omega_1 + \dots + w_d \omega_d)^n$. By the multinomial theorem, $a_{\alpha} = \underline{w}^\alpha$, where $\underline{w}=(w_1,\dots,w_d)$. Therefore, $a_{\alpha} a_{\beta} = \underline{w}^\alpha \underline{w}^\beta = \underline{w}^{\alpha+\beta}$. Since $\alpha+\beta = \gamma +\delta$, the equality $a_{\alpha} a_{\beta} = a_{\gamma} a_{\delta}$ holds in $\bar{K}$. Note that $a_{\alpha} \in R$ for all $|\alpha|=n$, we obtain the desired equality.
\end{proof}

\begin{rem}\label{rem_isom}
Recall that $T^*_R = \Spec \, R[\underline{t}]$ and ${\rm tot}(S^n \Omega_R^1) = \Spec \, R[t_\alpha, |\alpha| = n]$. The condition \eqref{eq_cond} defines a closed subscheme
\begin{align*}
    V_n(\Omega_R^1):=\Spec \, R[t_\alpha, |\alpha| = n] / \langle t_\alpha t_\beta - t_\gamma t_\delta, \alpha + \beta = \gamma + \delta \rangle \subseteq{\rm tot}(S^n \Omega_R^1).
\end{align*}
The spectral data $a: \Spec \, R \rightarrow S^n \Omega_R^1$ always factors through $V_n(\Omega_R^1)$. Moreover, there is a morphism
\begin{align*}
    R[t_\alpha, |\alpha|=n] \rightarrow R[\underline{t}], \quad t_\alpha \mapsto \underline{t}^\alpha,
\end{align*}
and the induced morphism
\begin{align*}
    T^*_R \rightarrow {\rm tot}(S^n \Omega_R^1), \quad v \mapsto v^n
\end{align*}
factors through $V_n(\Omega_R^1)$ \cite[\S 2]{ERT94}. 

There is a natural $\mu_n$-action on $T^*_R$, i.e.,
\begin{align*}
    \mu_n \times T^*_R \rightarrow T^*_R, \quad (\zeta,v) \mapsto \zeta v
\end{align*}
where $\mu_n$ is the cyclic group of order $n$ and $\zeta \in \mu_n$ is an $n$-th root of unity. Since this morphism $T^*_R \rightarrow V_n(\Omega_R^1)$ is invariant under the $\mu_n$-action, we obtain an induced morphism $T^*_R /\!/\mu_n \rightarrow V_n(\Omega_R^1)$, which is an isomorphism. 
\begin{center}
\begin{tikzcd}
\mu_n \times T^*_R \arrow[r] \arrow[d] & T^*_R \arrow[d] \arrow[rdd, bend left] & \\
T^*_R \arrow[r] \arrow[rrd, bend right] & T^*_R /\!/\mu_n \arrow[rd, dotted] & \\
& & V_n(\Omega_R^1)
\end{tikzcd}
\end{center}
Indeed, it is easy to check that
\begin{align*}
    T^*_R /\!/\mu_n = \Spec \, R[\underline{t}]^{\mu_n} = \Spec \, R[\underline{t}^\alpha, |\alpha| = n],
\end{align*}
and the isomorphism 
\begin{align*}
    T^*_R /\!/\mu_n \rightarrow V_n(\Omega_R^1)
\end{align*}
is induced by
\begin{align*}
    R[t_\alpha, |\alpha| = n] / \langle t_\alpha t_\beta - t_\gamma t_\delta, \alpha + \beta = \gamma + \delta \rangle \rightarrow R[\underline{t}^\alpha, |\alpha| = n], \quad t_\alpha \mapsto \underline{t}^\alpha.
\end{align*}
\end{rem}

\begin{rem}\label{rem_UFD}
Since the regular local ring $R = \mathcal{O}_{X,x }$ is a unique factorization domain (UFD), the condition \eqref{eq_cond} implies that there exist $g,u_1,\dots,u_d \in R$ such that
\begin{align*}
    a_\alpha = g \underline{u}^\alpha.
\end{align*}
Since we always assume that $a$ is nonzero, we have $g \neq 0$. Note that the decomposition of $a_\alpha$ is not unique. For example, we can choose $c \in R^{\times}$ and define
\begin{align*}
    g' = c^{-n} g, \, u'_i = c u_i,
\end{align*}
and then $a_\alpha = g' \underline{u'}^\alpha$.
\end{rem}

We define a quotient $R$-algebra $R_a$ of $R[\underline{t}]$ by the equation $t^n-a=0$. More precisely, $R_a = R[\underline{t}]/ I$, where $I = \langle \underline{t}^\alpha - a_\alpha, \, |\alpha|= n \rangle$, and there is a natural morphism $R \rightarrow R_a$.

\begin{prop}\label{prop_dim_d}
Let $a = \sum_{|\alpha| = n} {n \choose \alpha} a_\alpha  \underline{\omega}^{\alpha}$ be as above. The following conditions are equivalent:
\begin{enumerate}
\item[$(1)$] $a$ satisfies the condition \eqref{eq_cond};
\item[$(2)$] $\dim R_a = d$;
\item[$(3)$] $a$ is cyclic spectral data.
\end{enumerate}
\end{prop}

\begin{proof}
We first prove that conditions $(1)$ and $(2)$ are equivalent. Assume first that $a$ satisfies the condition \eqref{eq_cond} and suppose that $a_\alpha = g \underline{u}^\alpha$ for some $g,u_1,\dots,u_d \in R$ by Remark \ref{rem_UFD}. We define a morphism
\begin{align*}
    \varphi: R_a = R[\underline{t}]/I \rightarrow R[s]/(s^n - g), \quad t_i \mapsto u_i s, 1 \leq i \leq d.
\end{align*}
Note that $\varphi|_{R} = \boldsymbol{1}$, and we obtain a natural injection
\begin{align*}
    R \hookrightarrow R[\underline{t}]/I.
\end{align*}
Since $I$ is generated by $\underline{t}^\alpha - a_{\alpha}$ for $|\alpha|=n$, $R[\underline{t}]/I$ is integral over $R$. Therefore, 
\begin{align*}
    \dim R[\underline{t}]/I = \dim R = d.
\end{align*}
Conversely, we assume that $\dim R_a = d$. There is a natural morphism $R \rightarrow R[\underline{t}]/I$. Clearly, $R[\underline{t}]/I$ is integral over $R$, and then, 
\begin{align*}
    \dim R[\underline{t}]/I = \dim R / (R \cap I).
\end{align*}
Since $\dim R[\underline{t}]/I = d$, we have $R \cap I=(0)$. Now we will show that 
\begin{align*}
    a_\alpha a_\beta - a_\gamma a_\delta \in R \cap I = (0), 
\end{align*}
when $\alpha+\beta = \gamma+\delta$, and thus, $a_\alpha a_\beta = a_\gamma a_\delta$. The equality $\underline{t}^\alpha \underline{t}^\beta = \underline{t}^\gamma \underline{t}^\delta$ in $R[\underline{t}]$ and the relations $\underline{t}^\alpha = a_\alpha \, {\rm mod} \, I$ in $R[\underline{t}]/I$ for $|\alpha|=n$ imply that $a_\alpha a_\beta - a_\gamma a_\delta \in I$. Note that the element $a_\alpha a_\beta - a_\gamma a_\delta$ also lies in $R$. Therefore, $a_\alpha a_\beta - a_\gamma a_\delta \in R \cap I$.

Now we consider the equivalence of conditions $(1)$ and $(3)$. Lemma \ref{lem_cond_*} shows that $(3)$ implies $(1)$, and we only have to prove the other direction. There is a natural morphism
\begin{align*}
    T^*_R \rightarrow {\rm Chow}_n(T^*_R/R), \quad v \mapsto [v,\zeta v, \dots, \zeta^{n-1} v],
\end{align*}
where $\zeta$ is a primitive $n$-th root of unity. Clearly, this morphism factors through $T^*_R /\!/\mu_n$. By Remark \ref{rem_isom}, the isomorphism  $T^*_R /\!/\mu_n \rightarrow V_n(\Omega_R^1)$ makes the following diagram commute.
\begin{center}
\begin{tikzcd}
T^*_R /\!/\mu_n \arrow[r] \arrow[d, "\cong"] & {\rm Chow}_n(T^*_R/R) \arrow[d]\\
V_n(\Omega_R^1) \arrow[r] & {\rm tot}(S^n \Omega_R^1)
\end{tikzcd}
\end{center}
Now assume that $a$ satisfies the condition \eqref{eq_cond}, and then, $a$ factors through $V_n(\Omega_R^1)$.
\begin{center}
\begin{tikzcd}
& T^*_R /\!/\mu_n \arrow[r] \arrow[d, "\cong"] & {\rm Chow}_n(T^*_R/R) \arrow[d]\\
\Spec \, R \arrow[r] \arrow[rr, bend right =15, "a" description] & V_n(\Omega_R^1) \arrow[r] & {\rm tot}(S^n \Omega_R^1)
\end{tikzcd}
\end{center}
Therefore, $a$ is spectral data.
\end{proof}

\begin{rem}
If $a$ is nonzero cyclic spectral data, it is clear that $\dim R_a =d$ by the construction given by Chen and Ng\^o. Proposition \ref{prop_dim_d} provides another proof from the viewpoint of commutative algebras.
\end{rem}

\begin{exmp}\label{exmp_dim_2}
We give an example illustrating the proof given in Proposition \ref{prop_dim_d}. For simplicity, we consider the case of affine spaces. Although $R$ is assumed to be a regular local ring in the above, the properties needed in the arguments are smoothness and the UFD property, of which the affine space automatically satisfies.

Let $d=n=2$ and let $R=k[x_1,x_2]$. We have
\begin{align*}
    R_a = k[x_1,x_2,t_1,t_2]/I, 
\end{align*}
where the ideal $I$ is generated by elements
\begin{align*}
    t_1^2 - a_{(2,0)}, \, t_1 t_2 - a_{(1,1)}, \, t_2^2 - a_{(0,2)},
\end{align*}
where $a_{(2,0)}, a_{(1,1)}, a_{(0,2)} \in k[x_1,x_2]$, and the condition \eqref{eq_cond} includes a single relation 
\begin{align*}
    a^2_{(1,1)} = a_{(2,0)} a_{(0,2)}.
\end{align*}

Suppose that $a$ satisfies the condition \eqref{eq_cond}, and then 
\begin{align*}
    a_{(2,0)} = g u_1^2, \, a_{(1,1)} = g u_1 u_2, \, a_{(0,2)} = g u_2^2
\end{align*}
for some $g,u_1, u_2 \in k[x_1,x_2]$ by Remark \ref{rem_UFD}. The morphism
\begin{align*}
    \varphi: k[x_1,x_2,t_1,t_2]/I \rightarrow k[x_1,x_2,s]/(s^2 - g), \quad t_1 \mapsto u_1 s, \, t_2 \mapsto u_2 s
\end{align*}
is well-defined and induces a natural inclusion
\begin{align*}
    k[x_1,x_2] \hookrightarrow k[x_1,x_2,t_1,t_2]/I \rightarrow k[x_1,x_2,s]/(s^2 - g).
\end{align*}
Since $k[x_1,x_2,t_1,t_2]/I$ is integral over $k[x_1,x_2]$ and $k[x_1,x_2] \hookrightarrow k[x_1,x_2,t_1,t_2]/I$ is injective, we have $\dim R_a =2$.

For the other direction, we suppose that $\dim R_a = 2$. Motivated by the identity $(t_1 t_2)^2 - (t_1^2) (t_2^2) = 0$, we obtain $a^2_{(1,1)} - a_{(2,0)} a_{(0,2)} \in I$. Indeed, $I \cap k[x_1,x_2] = \langle a^2_{(1,1)} - a_{(2,0)} a_{(0,2)} \rangle$. Therefore, $k[x_1,x_2,t_1,t_2]/I$ is finitely generated over $k[x_1,x_2] / (I \cap k[x_1,x_2])$. As an integral extension, we have 
\begin{align*}
    \dim R_a = \dim k[x_1,x_2] / (I \cap k[x_1,x_2]).
\end{align*}
If $a^2_{(1,1)} - a_{(2,0)} a_{(0,2)} \neq 0$, then $\dim k[x_1,x_2] / (I \cap k[x_1,x_2]) \leq 1$. This contradicts the assumption that $\dim R_a =2$. Therefore, $a^2_{(1,1)} - a_{(2,0)} a_{(0,2)} = 0$.
\end{exmp}

\begin{exmp}\label{exmp_spectral_cover_2}
We continue the discussion in Example \ref{exmp_dim_2} and consider the case $d=n=2$ and $R=k[x_1,x_2]$. The spectral data 
\begin{align*}
    a = a_{(2,0)} \omega_1^2 + 2 a_{(1,1)}\omega_1 \omega_2 + a_{(0,2)} \omega_2^2,
\end{align*}
which satisfies $a^2_{(1,1)} = a_{(2,0)} a_{(0,2)}$, gives a spectral cover $R_a = k[x_1,x_2,t_1,t_2]/ I$, where $I$ is generated by 
\begin{align*}
    t_1^2 - a_{(2,0)}, \, t_1 t_2 - a_{(1,1)}, \, t_2^2 - a_{(0,2)}.
\end{align*}
Note that $\dim R_a =2$ by Proposition \ref{prop_dim_d}.
\begin{enumerate}
\item Define
\begin{align*}
    a_{(2,0)} =x_1^2, \, a_{(1,1)} = x_1 x_2, \, a_{(0,2)} = x_2^2. 
\end{align*}
In this case, $I$ is generated by 
\begin{align*}
    t_1^2 - x_1^2, \, t_1 t_2 - x_1x_2, \, t_2^2 - x_2^2.
\end{align*}
Consider the nonzero element $(x_1 t_2 - x_2 t_1) + I \in k[\underline{x},\underline{t}]/ I$. We have
\begin{align*}
    x_i (x_1 t_2 - x_2 t_1) = 0 \, {\rm mod} \, I , \, t_i (x_1 t_2 - x_2 t_1) = 0 \, {\rm mod} \, I
\end{align*}
for $i=1,2$. Therefore, the maximal ideal $\mathfrak{n} = \langle x_1,x_2,t_1,t_2 \rangle \subseteq k[x_1,x_2,t_1,t_2]/ I$, and then any element in this maximal ideal is a zero divisor. The depth of the local ring $(k[x_1,x_2,t_1,t_2]/ I)_{\mathfrak{n}}$ is $0$, while its dimension is $2$. Therefore, the $R$-algebra $R_a$ is not Cohen-Macaulay.

\item Define
\begin{align*}
    a_{(2,0)} = a_{(1,1)} = a_{(0,2)} = 1.
\end{align*}
The ideal $I$ is generated by 
\begin{align*}
    t_1^2 - 1, \, t_1 t_2 - 1, \, t_2^2 - 1.
\end{align*}
Then,
\begin{align*}
    k[\underline{x},\underline{t}]/ I \cong k[\underline{x}] \times k[\underline{x}],
\end{align*}
which is Cohen-Macaulay.
\end{enumerate}
\end{exmp}

In Proposition \ref{prop_dim_d}, we define a morphism 
\begin{align*}
\varphi: R[\underline{t}]/I \rightarrow R[s]/(s^n - g), \quad t_i \mapsto u_i s, \, 1 \leq i \leq d.
\end{align*}
Indeed, the kernel of this morphism is exactly the nilradical $\sqrt{(0)}$ of $R[\underline{t}]/I$.

\begin{lem}\label{lem_nilrad}
The kernel ${\rm ker} \, \varphi$ is the nilradical $\sqrt{(0)}$ of $R[\underline{t}]/I$.
\end{lem}

\begin{proof}
Since $R[s]/(s^n - g)$ is reduced, we have $\varphi(f(\underline{t})+I) = 0$ for any nilpotent element $f(\underline{t}) +I \in R[\underline{t}]/I$. Now, given any element $f(\underline{t}) + I \in {\rm ker}\, \varphi$, we write
\begin{align*}
    f(\underline{t})= f_0(\underline{t}) + \dots + f_{n-1}(\underline{t}), 
\end{align*}
where $f_i(\underline{t})$ is a homogeneous polynomial of degree $i$ with respect to the variables $\underline{t}$. Since $\varphi(f(\underline{t})+I) = 0$, we have $\varphi(f_i(\underline{t}) +  I) = 0$ for each $i$. Then,
\begin{align*}
    \varphi(f_i(\underline{t})+I) = s^i f_i(\underline{u}) + (s^n -g),
\end{align*}
which implies that $f_i(\underline{u}) = 0$ for each $i$. Since $I$ is generated by $\underline{t}^\alpha - a_\alpha$ for $|\alpha|=n$, where $a_\alpha \in R$, we have $f_i(\underline{t})^n + I = r+I$ for some $r \in R$. Note that 
\begin{align*}
    \varphi(f_i(\underline{t})^n + I) = s^{in} f_i(\underline{u})^n + (s^n -g) = g^i f_i(\underline{u})^n + (s^n -g). 
\end{align*}
Therefore, 
\begin{align*}
    f_i(\underline{t})^n + I = g^i f_i(\underline{u})^n + I = 0 + I
\end{align*}
for each $i$, and $f(\underline{t})+I$ is nilpotent.
\end{proof}

\begin{rem}\label{rem_generic_fiber}
Let $K$ be the fraction field of $R$, and consider $K \otimes_R R[\underline{t}]/ I \cong K[\underline{t}]/I$. We omit the ideal $I$ from the notation for simplicity. Since we suppose that the given cyclic spectral data $a$ is nonzero, there exists an index $i$ such that $u_i \neq 0$. From the relations $\underline{t}^\alpha = g \underline{u}^\alpha$, we have 
\begin{align*}
    t^n_i = g u^n_i, \, t_i^{n-1} t_j = g u_i^{n-1}u_j.
\end{align*}
The first equation implies that $t_i$ is invertible and these two equations give
\begin{align*}
    \frac{u_j}{u_i} = \frac{t_j}{t_i}.
\end{align*}
In view of this, it is easy to see that 
\begin{align*}
    K \otimes_R R[\underline{t}]/ I \cong K \otimes_R R[s]/(s^n -g).
\end{align*}
Moreover, we have $(u_i t_j - u_j t_i)^n = 0 \, {\rm mod} \, I$, and then $u_i t_j - u_j t_i \in \sqrt{(0)}$. 
\end{rem}

Lemma \ref{lem_nilrad} implies that the morphism $\varphi$ factors through $(R[\underline{t}]/I)/ \sqrt{(0)}$, i.e.,
\begin{align*}
    \varphi: R[\underline{t}]/I \twoheadrightarrow (R[\underline{t}]/I)/ \sqrt{(0)} \hookrightarrow R[s]/(s^n - g).
\end{align*}
For convenience, we introduce the following notations
\begin{align*}
S:=R_a = R[\underline{t}]/I, \, S_{\rm red}= (R[\underline{t}]/I)/ \sqrt{(0)}, \, S' = R[s]/(s^n - g).
\end{align*}

\begin{prop}\label{prop_local_surj}
Given cyclic spectral data $a = \sum_{|\alpha| = n} {n \choose \alpha} a_\alpha \underline{\omega}^{\alpha} \in H^0(\Spec \, R, S^n \Omega_R^1)$, $S'$ is a maximal Cohen-Macaulay $S$-module of generic rank one.
\end{prop}

\begin{proof}
We first prove that $S'$ is a maximal Cohen-Macaulay $S$-module, and it is enough to prove that for any prime ideal $\mathfrak{q} \subseteq S$, we have
\begin{align*}
    {\rm depth}_{S_\mathfrak{q}} S'_\mathfrak{q} = \dim S_\mathfrak{q}.
\end{align*}
Let $\mathfrak{p}:= \mathfrak{q} \cap R$. Since $R_{\mathfrak{p}}$ is a regular local ring, we choose a regular system of parameters $x_1,\dots,x_h$ of $R_{\mathfrak{p}}$, where $h = \dim R_{\mathfrak{p}}$. Since $S' = R[s]/(s^n-g)$ is a free $R$-module of rank $n$, this sequence is also an $R_\mathfrak{p}$-regular sequence of $S'_\mathfrak{q}$ and
\begin{align*}
    {\rm depth}_{R_\mathfrak{p}} S'_\mathfrak{q} = h.
\end{align*}
Since $R \hookrightarrow S$ is injective, we have
\begin{align*}
    {\rm depth}_{S_\mathfrak{q}} S'_\mathfrak{q} \geq {\rm depth}_{R_\mathfrak{p}} S'_\mathfrak{q} = h.
\end{align*}
For dimensions, since $S_\mathfrak{q}$ is finite over $R_\mathfrak{p}$, we have $\dim S_\mathfrak{q} \leq h$. Since the depth is always bounded above by the dimension
\begin{align*}
h \leq   {\rm depth}_{S_\mathfrak{q}} S'_\mathfrak{q} \leq \dim S_\mathfrak{q} \leq h,
\end{align*}
we have
\begin{align*}
    {\rm depth}_{S_\mathfrak{q}} S'_\mathfrak{q} = h = \dim S_\mathfrak{q},
\end{align*}
and $S'$ is a maximal Cohen-Macaulay $S$-module.

Now we will show that $S'$ is an $S$-module of generic rank one. In Remark \ref{rem_generic_fiber}, we show that 
\begin{align*}
    S_K \cong K[s]/(s^n-g) = S'_K,
\end{align*}
where $K = {\rm Frac}\, R$. Since $g \neq 0$ and $s^n-g$ is separable, $S_K$ is a finite product of fields. Let $\mathfrak{q}$ be a minimal prime of $S$. Then, $S_\mathfrak{q}$ is one of the field factors of $S_K$. Under the isomorphism $S_K \cong S'_K$, we have $S_\mathfrak{q} \cong S'_\mathfrak{q}$. Therefore, $S'$ is of generic rank one as an $S$-module.
\end{proof}

\begin{rem}\label{rem_cons_Hit}
Given spectral data $a = \sum_{|\alpha| = n} {n \choose \alpha} a_\alpha \underline{\omega}^{\alpha}$, there exist $g,u_1,\dots,u_d \in R$ such that $a_\alpha = g \underline{u}^\alpha$ for $|\alpha| = n$. Then, $S' = R[s]/(s^n-g)$ and the morphism 
\begin{align*}
    R[\underline{t}]/I \rightarrow R[s]/(s^n-g)
\end{align*}
defined by $t_i \mapsto u_i s$ induces an $R[\underline{t}]$-module structure on $R[s]/(s^n-g)$. Regarding $R[s]/(s^n-g)$ as a free $R$-module with basis $\{1,s,\dots,s^{n-1}\}$, the Higgs field $\theta$ on
\begin{align*}
    R[s]/(s^n-g) \cong R \cdot 1 \oplus \dots \oplus R \cdot s^{n-1}
\end{align*}
induced by the $R[\underline{t}]$-module structure on $R[s]/(s^n-g)$ is
\begin{align*}
    \theta & = 
\begin{pmatrix}
     & u_1 & & \\
    &  & \ddots & \\
    & & & u_1 \\
    u_1 g & & & \\
\end{pmatrix}\omega_1
+ \cdots + 
\begin{pmatrix}
     & u_d & & \\
    &  & \ddots & \\
    & & & u_d \\
    u_d g & & & \\
\end{pmatrix}\omega_d \\
& =\begin{pmatrix}
     & 1 & & \\
    &  & \ddots & \\
    & & & 1 \\
    g & & & \\
\end{pmatrix} (\sum_{i=1}^d u_i \omega_i).
\end{align*}
\end{rem}

\begin{rem}
The ring $S'$ may not be normal, for instance, $S' = k[x,s]/(s^4-x^2)$. Moreover, $S'$ is normal if and only if $g$ is square-free. 
\end{rem}

\begin{cor}\label{cor_nor_MCM}
The normalization $\widetilde{S}$ of $S_{\rm red}$ is a maximal Cohen-Macaulay $S$-module of generic rank one.
\end{cor}

\begin{proof}
Note that $S_{\rm red} \cong R[\underline{u}s]/(s^n -g) \subseteq S' = R[s]/(s^n-g)$. The normalization $\widetilde{S}$ of $S_{\rm red}$ is also the normalization of $S'$. Since $R$ is a UFD, we write $g = u \prod_j p_j^{e_j}$, where $u$ is a unit and $p_j$ are irreducible elements. Define \begin{align*}
    q_i:=\prod_j p_j^{\lfloor \frac{ie_j}{n} \rfloor}
\end{align*}
for $0 \leq i \leq n-1$. The elements $z_i:=\frac{s^i}{q_i}$ are integral over $S_{\rm red}$. In fact, one has $\widetilde{S} = \bigoplus_{i=0}^{n-1} R z_i$. Therefore, the normalization $\widetilde{S}$ is a free $R$-module of rank $n$. The result then follows by an argument similar to the proof of Proposition \ref{prop_local_surj}.
\end{proof}

\begin{rem}
There is a natural $\mu_n$-action on $S'$, which is defined as
\begin{align*}
    s \rightarrow \zeta s,
\end{align*}
where $\zeta \in \mu_n$. Then, the structure of the normalization $\widetilde{S}$ can be understood as the eigenspace decomposition under the induced $\mu_n$-action on $\widetilde{S}$, where the direct summand $Rz_i$ is exactly the eigenspace with eigenvalue $\zeta^i$, and $\zeta$ is a primitive $n$-th root of unity.
\end{rem}

Now we will prove the existence of a maximal Cohen-Macaulay sheaf of generic rank one on the spectral cover. We first fix some notation. Let $\pi_a: X_a \rightarrow X$ be the spectral cover determined by spectral data $a$, let $(X_a)_{\rm red}$ be its reduced scheme, and let $\widetilde{X}_a$ be the normalization of $(X_a)_{\rm red}$.

\begin{thm}\label{thm_nor_MCM}
Let $a \in H^0(X, S^n \Omega_X^1)$ be nonzero cyclic spectral data. There exists a maximal Cohen-Macaulay sheaf of generic rank one on the spectral cover $X_a$. In particular, the Hitchin fiber $h^{-1}_X(a)$ is non-empty.
\end{thm}

\begin{proof}
Let $\nu: \widetilde{X}_a \rightarrow (X_a)_{\rm red} \rightarrow X_a$ be the composition. We will prove that $\nu_* \mathcal{O}_{\widetilde{X}_a}$ is a maximal Cohen-Macaulay $\mathcal{O}_{X_a}$-module of generic rank one. As proved in Proposition \ref{prop_local_surj} and Corollary \ref{cor_nor_MCM}, we have $(\nu_* \mathcal{O}_{\widetilde{X}_a})_\eta \cong \mathcal{O}_{X_a,\eta}$ at every minimal point $\eta$ of $X_a$. Therefore, $\nu_* \mathcal{O}_{\widetilde{X}_a}$ is of generic rank one. It remains to show that $\nu_* \mathcal{O}_{\widetilde{X}_a}$ is a maximal Cohen-Macaulay $\mathcal{O}_{X_a}$-module. Since Cohen-Macaulayness is a local property, it is enough to prove that $(\nu_*\mathcal{O}_{\widetilde{X}_a})_{y}$ is a maximal Cohen-Macaulay $\mathcal{O}_{X_a,y}$-module for any point $y \in X_a$. We choose an affine open subset $U \hookrightarrow X$, on which $\Omega_X^1$ is free. Let $V:=\pi_a^{-1}(U) \subseteq X_a$, where $\pi_a: X_a \rightarrow X$ is the spectral cover morphism. Then, $\mathcal{O}_V \cong \mathcal{O}_U[t_1,\dots,t_d]/ I'$, where $I'$ is the ideal determined by the spectral data $a$ on $U$ as in Proposition \ref{prop_dim_d}. We take a point $y \in V \subseteq X_a$ and denote by $\mathfrak{n}$ the corresponding prime ideal in $\mathcal{O}_V$. Let $x:=\pi_a(y) \in U \subseteq X$ and denote by $\mathfrak{m} \subseteq \mathcal{O}_U$ the corresponding prime ideal. We follow the same notation as above, especially those in Proposition \ref{prop_local_surj}. We define
\begin{align*}
    R:=\mathcal{O}_{X,x} = \mathcal{O}_{U,x}, \, I:=I'_{\mathfrak{m}}, \, S:=R[\underline{t}]/I,
\end{align*}
and let $\widetilde{S}$ be the normalization of $S_{\rm red}$. We have 
\begin{align*}
    \mathcal{O}_{X_a,y} =\mathcal{O}_{V,y} \cong (\mathcal{O}_{U}[\underline{t}]/I')_{y} = (R[\underline{t}]/I)_{\mathfrak{n}} = S_{\mathfrak{n}}.
\end{align*}
Moreover, a similar argument gives
\begin{align*}
    (\nu_*\mathcal{O}_{\widetilde{X}_a})_{y} \cong \widetilde{S}_{\mathfrak{n}}.
\end{align*}
By Proposition \ref{prop_local_surj} and Corollary \ref{cor_nor_MCM}, the theorem follows directly.
\end{proof}

\section{Hitchin Section}\label{sect_Hit_sect}
In this section, we will provide a global version of the $R$-algebra $S'=R[s]/(s^n-g)$ constructed in Proposition \ref{prop_dim_d}, and then with the aim of defining a section of the Hitchin morphism based on Remark \ref{rem_cons_Hit}. The global construction is a generalization of He--Liu's work \cite[\S3]{HL24}, and we first recall the definition of symmetric differentials of rank one \cite{BdO11}.
\begin{defn}
A symmetric differential $\omega \in H^0(X, S^m \Omega_X^1)$ is of \emph{rank one} if for every point there is an open subset $U \subseteq X$, a section $\mu \in H^0(U, \Omega_U^1)$ and a function $f \in \mathcal{O}_X(U)$ such that $\omega|_U  = f \mu^m$.
\end{defn}

\begin{prop}\label{prop_sym_diff}
Every cyclic spectral data is a symmetric differential of rank one.
\end{prop}

\begin{proof}
We follow the notations of Lemma \ref{lem_cond_*}. Let $K:=K(X)$ be the function field with algebraic closure $\bar{K}$. Let $a \in H^0(X,S^n \Omega_X^1)$ be cyclic spectral data, and we suppose that $a$ is nonzero. There exists $v_{\bar{K}} \in T^*_{\bar{K}}$ such that $a_{\bar{K}} = v_{\bar{K}}^n$. In Remark \ref{rem_isom}, we defined a closed embedding
\begin{align*}
    T^*_K \hookrightarrow {\rm tot}(S^n \Omega^1_K), \quad v \mapsto v^n. 
\end{align*}
Consider the induced closed immersion on projective spaces
\begin{align*}
\mathbb{P}(T^*_K) \hookrightarrow \mathbb{P}(S^n \Omega^1_K).
\end{align*}
The equation $a_{\bar{K}} = v_{\bar{K}}^n$ implies that $[a_{\bar{K}}]$ lies in the image of the base change of the morphism $\mathbb{P}(T^*_K) \hookrightarrow \mathbb{P}(S^n \Omega^1_K)$ to $\bar{K}$. Since the morphism  $\mathbb{P}(T^*_K) \hookrightarrow \mathbb{P}(S^n \Omega^1_K)$ is a closed immersion defined over $K$, there exists $v \in T^*_K$ such that $[a_K] = [v^n]$. Therefore, there exists $c \in K^\times$ such that $a_K = c v^n$, which is a generic decomposition of spectral data $a_K$.

Now we will show that the generic decomposition $a_K = c v^n$ induces a global decomposition. Regarding $v$ as a section of $\Omega_K^1$, define a subsheaf $\ell_v := Kv \subseteq \Omega_K^1$. Consider its saturation
\begin{align*}
    \mathscr{L}:= \Omega^1_X \cap \ell_v,
\end{align*}
where the intersection is taken inside $\Omega_K^1$, and it is a rank one reflexive sheaf on $X$. Since $X$ is smooth, $\mathscr{L}$ is a line bundle, and we obtain a global section $\alpha: \mathscr{L} \rightarrow \Omega_X^1$ induced by $v$, i.e., $\alpha_K = v$. Now $\tau_K:=a/\alpha^n$ is a rational section of $\mathscr{L}^n$. Since $\alpha(\mathscr{L})$ is saturated, any pole of $\tau_K$ in codimension one will give a pole of $a$. Since $a$ is regular, $\tau_K$ does not have any pole in codimension one. As $\mathscr{L}^n$ is locally free on the smooth variety $X$, $\tau_K$ extends uniquely to a section $\tau$ of $\mathscr{L}^n$ by the normality of $X$. Therefore, we obtain a global decomposition $a = \tau \alpha^n$, which implies that the cyclic spectral data $a$ is a symmetric differential of rank one. 
\end{proof}

The proof of Proposition \ref{prop_sym_diff} also implies the following corollary.

\begin{cor}\label{cor_decomp}
Let $a$ be nonzero cyclic spectral data. There exists a triple $(\mathscr{L},\tau,\alpha)$, where $\mathscr{L}$ is a line bundle on $X$, $\tau \in H^0(X,\mathscr{L}^n)$ and $\alpha \in H^0(X,\Omega_X^1 \otimes \mathscr{L}^{-1})$, such that $a = \tau \alpha^n$ and $\alpha$ does not vanish in codimension one. The isomorphism class of $\mathscr{L}$ is unique, and the pair $(\tau,\alpha)$ is unique up to $(\tau,\alpha) \sim (c^{-n} \tau, c \alpha)$ for some $c \in H^0(X,\mathcal{O}_X^\times)$. Moreover, the image $\alpha(\mathscr{L})$ is saturated in $\Omega_X^1$.
\end{cor}

\begin{rem}
We would like to mention that the proof of Proposition \ref{prop_sym_diff} is inspired by Bogomolov--De Oliveira and He--Liu's work \cite{BdO11,HL24}. If we know that $a$ is a symmetric differential of rank one, the existence and uniqueness of the line bundle $\mathscr{L}$ and the decomposition $a = \tau \alpha^n$ are direct results of Bogomolov and De Oliveira \cite[Proposition 2.2]{BdO11}. Moreover, He--Liu also provided a proof when $n=2$ \cite[Proposition 3.6]{HL24}, which can be directly generalized to arbitrary $n$. Furthermore, the statement that $\alpha$ is saturated in $\Omega_X^1$ is also proved in \cite[Lemma 3.8]{HL24}.   
\end{rem}

With the help of Corollary \ref{cor_decomp}, we can construct a Higgs bundle for each nonzero cyclic spectral data $a$. This gives a second proof of the nonemptiness of the Hitchin fiber.

\begin{proof}[Second proof of Theorem \ref{thm_nonempty}]
Let $a$ be nonzero cyclic spectral data. By Corollary \ref{cor_decomp}, there exists a line bundle $\mathscr{L}$ such that $a=\tau \alpha^n$, where $\tau \in H^0(X,\mathscr{L}^n)$ and $\alpha \in H^0(X,\Omega_X^1 \otimes \mathscr{L}^{-1})$. Now we construct a Higgs bundle $(E_a,\theta_a)$ as follows. We define a locally free sheaf
\begin{align*}
    E_a := \mathcal{O}_X \oplus \mathscr{L}^{-1} \oplus \dots \oplus \mathscr{L}^{-(n-1)}
\end{align*}
with a Higgs field
\begin{align*}
\theta_a = 
\begin{pmatrix}
     & \alpha & & \\
    &  & \ddots & \\
    & & & \alpha \\
    \tau \alpha & & & \\
\end{pmatrix}.
\end{align*}
Clearly, $h_X( (E_a,\theta_a)  ) = (0,\dots,0,-a)$, and the Hitchin fiber $h^{-1}_X(a)$ is non-empty.

\end{proof}

\begin{rem}
The construction of the above Higgs bundle $(E_a,\theta_a)$ is indeed a global version of the Higgs bundle constructed in Remark \ref{rem_cons_Hit}, where $\tau$ and $\alpha$ correspond to $g$ and $\sum_{i=1}^d u_i \omega_i$ respectively. Moreover, the decomposition of spectral data $a$ on the regular local ring $R:=\mathcal{O}_{X,x}$ provides another viewpoint that links cyclic spectral data to the symmetric differentials of rank one.
\end{rem}

The above construction of $E_a$ gives a scheme $X'_a$ together with a natural morphism $\pi'_a: X'_a \rightarrow X$. More precisely, we define an $\mathcal{O}_X$-algebra
\begin{align*}
    \mathcal{A}_\tau := \oplus_{i=0}^{n-1} \mathscr{L}^{-i},
\end{align*}
where the multiplication is defined as follows
\begin{align*}
\mathscr{L}^{-i} \otimes \mathscr{L}^{-j} \rightarrow 
\begin{cases}
\mathscr{L}^{-(i+j)} & \text{ if } i+j<n, \\
(\boldsymbol{1} \otimes \tau)(\mathscr{L}^{-(i+j-n)} \otimes \mathscr{L}^{-n}) & \text{ if } i+j \geq n.
\end{cases}
\end{align*}
We define $X'_a := \Spec \, \mathcal{A}_\tau$, and $\alpha:\mathscr{L} \rightarrow \Omega_X^1$ induces a morphism $X'_a \rightarrow T^*_X$. Moreover, the construction of $X'_a$ does not depend on the choice of the decomposition $a=\tau \alpha^n$. In fact, suppose that $(\tau',\alpha') = (c^{-n} \tau, c \alpha)$ for $c \in H^0(X,\mathcal{O}_X^\times)$, and then $\mathcal{A}_\tau \cong \mathcal{A}_{\tau'}$, which are compatible with corresponding morphisms to $T^*_X$.

\begin{lem}\label{lem_X'a}
The scheme $X'_a:= \Spec \, \mathcal{A}_\tau$ satisfies the following properties:
\begin{enumerate}
\item $\pi'_a: X'_a \rightarrow X$ is finite flat of degree $n$;
\item $X'_a$ is reduced and Cohen-Macaulay;
\item the pair $(\tau,\alpha)$ induces a finite morphism $X'_a \rightarrow X_a$ whose scheme-theoretic image is $(X_a)_{\rm red}$;
\item the induced morphism $X'_a \rightarrow (X_a)_{\rm red}$ is an isomorphism at the generic points. In particular, it is finite and birational.
\end{enumerate}
\end{lem}

\begin{proof}
Trivializing $\mathscr{L}$, the assertion follows directly from Lemma \ref{lem_nilrad} and Proposition \ref{prop_local_surj}.
\end{proof}

\begin{thm}\label{thm_BNR_corr}
Let $a$ be nonzero cyclic spectral data. There is an equivalence of the category of Higgs bundles on $X$ with spectral data $a$ and the category of maximal Cohen-Macaulay sheaves of generic rank one on $X'_a$. Furthermore, this correspondence gives an isomorphism between the Hitchin fiber $h^{-1}_X(a)$ and the stack $\mathcal{MCM}_1(X'_a)$ of maximal Cohen-Macaulay sheaves of generic rank one on $X'_a$.
\end{thm}

\begin{proof}
By Corollary \ref{cor_decomp}, we choose a triple $(\mathscr{L},\tau,\alpha)$ corresponding to the given spectral data $a$. Let $M$ be a maximal Cohen-Macaulay sheaf of generic rank one on $X'_a$. Define $E:=(\pi'_a)_*(M)$. We have
\begin{align*}
    \theta: E & = (\pi'_a)_*(M) \rightarrow (\pi'_a)_*(M \otimes (\pi'_a)^* \mathscr{L}) \\
    & \rightarrow (\pi'_a)_*(M) \otimes \mathscr{L} \xrightarrow{1 \otimes \alpha} (\pi'_a)_*(M) \otimes \Omega_X^1 = E \otimes \Omega_X^1,
\end{align*}
where the first morphism is given by the tautological section. It is easy to check that $E$ is a locally free sheaf of rank $n$ and the morphism $\theta$ is a Higgs field with spectral data $a$.

For the other direction, we start with a Higgs bundle $(E,\theta)$ with spectral data $a$. We claim that $\theta$ factors through $E \otimes \mathscr{L} \xrightarrow{1 \otimes \alpha} E \otimes \Omega_X^1$
\begin{center}
\begin{tikzcd}
& E \otimes \mathscr{L} \arrow[rd, "1 \otimes \alpha"] & \\
E \arrow[rr, "\theta"] \arrow[ru, "\vartheta", dotted] & & E \otimes \Omega_X^1 
\end{tikzcd}
\end{center}
such that $\vartheta^n = \tau \boldsymbol{1}$. Then, the morphism $\vartheta: E \rightarrow E \otimes \mathscr{L}$ defines a $\mathcal{O}_{X'_a}$-structure on $E$, and we obtain a coherent sheaf $M$ on $X'_a$. By an argument similar to those in the proofs of Proposition \ref{prop_local_surj} and Theorem \ref{thm_nor_MCM}, $M$ is a maximal Cohen-Macaulay $\mathcal{O}_{X'_a}$-module of generic rank one. 

The claim that $\theta$ factors through $E \otimes \mathscr{L} \xrightarrow{1 \otimes \alpha} E \otimes \Omega_X^1$ is a generalization of \cite[Lemma 4.5]{HL24}. We choose an $n$-th root $\sqrt[n]{\tau_{\bar{K}}}$ of $\tau_{\bar{K}}$, where $\bar{K}$ is the algebraic closure of the fraction field $K(X)$. Then, the generic cycle of the spectral data $a_{\bar{K}} = \tau_{\bar{K}} \alpha^n_{\bar{K}}$ is
\begin{align*}
    (\sqrt[n]{\tau_{\bar{K}}}\alpha_{\bar{K}} ,  \zeta\sqrt[n]{\tau_{\bar{K}}}\alpha_{\bar{K}}, \dots, \zeta^{n-1}\sqrt[n]{\tau_{\bar{K}}}\alpha_{\bar{K}} ),
\end{align*}
where $\zeta$ is a primitive $n$-th root of unity. Then, the composite
\begin{align*}
    E \xrightarrow{\theta} E \otimes \Omega_X^1 \rightarrow E \otimes (\Omega_X^1 / \alpha(\mathscr{L}))
\end{align*}
vanishes generically. Since $\alpha(\mathscr{L})$ is saturated, the quotient $(\Omega_X^1 / \alpha(\mathscr{L}))$ is torsion-free. Therefore, the composite vanishes, and $\theta$ factors through $E \otimes \mathscr{L} \xrightarrow{1 \otimes \alpha} E \otimes \Omega_X^1$. Furthermore, since the generic eigenvalues of $\vartheta$ are 
\begin{align*}
     \sqrt[n]{\tau_{\bar{K}}}, \zeta\sqrt[n]{\tau_{\bar{K}}}, \dots,  \zeta^{n-1}\sqrt[n]{\tau_{\bar{K}}},
\end{align*}
we have $\vartheta^n = \tau \boldsymbol{1}$.
\end{proof}

\begin{rem}
If $X$ is an algebraic surface and $X'_a$ is normal, the category of maximal Cohen-Macaulay sheaves on $X'_a$ is equivalent to the category of reflexive sheaves on $X'_a$. In this special case, the Hitchin fiber is isomorphic to the stack of reflexive sheaves of generic rank one on $X'_a$. Furthermore, He--Liu introduced a tower of Cohen-Macaulayfications of $X_a$, which depends on a discussion of $\tau$ and is helpful to check whether $X'_a$ is normal or not, and we refer the reader to \cite[\S 4.2.2]{HL24} for more details.
\end{rem}

We will follow the notations given in \cite[\S3]{HL24} and attempt to make the construction in the second proof into a section. To simplify notations, we define
\begin{align*}
    U_{\mathscr{L}}: = H^0(X,\mathscr{L}^n), \, W_{\mathscr{L}} := H^0(X, \Omega_X^1 \otimes \mathscr{L}^{-1}),
\end{align*}
and
\begin{align*}
    V_{\mathscr{L}}:= U_{\mathscr{L}} \times W_{\mathscr{L}}
\end{align*}
with a $\mathbb{G}_m$-action defined as
\begin{align*}
    c \cdot (\tau,\alpha) : = (c^{-n} \tau, c \alpha).
\end{align*}
Then we define a subset $W^{\rm sat}_\mathscr{L} \subseteq W_\mathscr{L}$ consisting of sections $\alpha$ that do not vanish in codimension one, which also implies that the image $\alpha(\mathscr{L})$ in $\Omega_X^1$ is saturated (Corollary \ref{cor_decomp}), and we introduce
\begin{align*}
   V_{\mathscr{L}}^{\rm sat,\times} :=  (U_{\mathscr{L}} \backslash \{0\} ) \times W^{\rm sat}_\mathscr{L} \subseteq V_{\mathscr{L}}.
\end{align*}
The induced $\mathbb{G}_m$-action on $V_{\mathscr{L}}^{\rm sat,\times}$ is well-defined. 

\begin{defn}
Define
\begin{align*}
    \mathcal{P}_X := \{ \mathscr{L} \in {\rm Pic}(X) \, | \,  U_\mathscr{L} \neq 0, \, W^{\rm sat}_\mathscr{L} \neq 0  \}.
\end{align*}
\end{defn}

For each $\mathscr{L} \in \mathcal{P}_X$, we have a natural morphism
\begin{align*}
V_{\mathscr{L}}^{\rm sat,\times} \rightarrow \mathscr{B}^{\rm cyc,\times}_{X}, \quad (\tau,\alpha) \mapsto \tau \alpha^n,
\end{align*}
which is compatible with the $\mathbb{G}_m$-action. Then, we obtain a morphism
\begin{align*}
    \bar{q}_{\mathscr{L}}: [V^{\rm sat,\times}_{\mathscr{L}} / \mathbb{G}_m] \rightarrow \mathscr{B}^{\rm cyc,\times}_{X}.
\end{align*}

\begin{thm}\label{thm_stack_Hit_sect}
For each $\mathscr{L} \in \mathcal{P}_X$, there exists a morphism $\bar{s}_{\mathscr{L}}: [V^{\rm sat,\times}_{\mathscr{L}} / \mathbb{G}_m] \rightarrow \mathscr{M}^{\rm cyc,\times}_X$ such that $\bar{q}_{\mathscr{L}} = {\rm sd}_X \circ \bar{s}_{\mathscr{L}}$, i.e.,
\begin{center}
\begin{tikzcd}
{[V_{\mathscr{L}}^{\mathrm{sat}, \times} / \mathbb{G}_m]} \arrow[r,"\bar{s}_\mathscr{L}", dotted] \arrow[rd,"\bar{q}_\mathscr{L}"] & \mathscr{M}^{\rm cyc,\times}_X \arrow[d, "{\rm sd}_X"] \\
& \mathscr{B}^{\rm cyc, \times}_X \, .
\end{tikzcd}
\end{center}
\end{thm}

\begin{proof}
We go back to the construction of the Higgs bundle in the second proof of Theorem \ref{thm_nonempty} and modify the notations slightly. Once a line bundle $\mathscr{L}$ is fixed, so is the underlying locally free sheaf and it is defined as 
\begin{align*}
    E_{\mathscr{L}}: = \mathcal{O}_X \oplus \mathscr{L}^{-1} \oplus \dots \oplus \mathscr{L}^{-(n-1)}.
\end{align*}
Then, we define a morphism
\begin{align*}
    V^{\rm sat,\times}_{\mathscr{L}} \rightarrow \mathscr{M}^{\rm cyc,\times}_X , \quad (\tau,\alpha) \mapsto (E_{\mathscr{L}},\theta_{\tau,\alpha}),
\end{align*}
where
\begin{align*}
\theta_{\tau,\alpha} = 
\begin{pmatrix}
     & \alpha & & \\
    &  & \ddots & \\
    & & & \alpha \\
    \tau \alpha & & & \\
\end{pmatrix}.
\end{align*}
Given any $c \in \Gamma(X,\mathcal{O}_X^\times)$, it is easy to check that the Higgs fields $\theta_{\tau,\alpha}$ and $\theta_{c^{-n}\tau,c\alpha}$ are conjugate via a diagonal matrix
\begin{align*}
D(c) = 
\begin{pmatrix}
    1 & & & \\
    & c & & \\
    & & \ddots & \\
    & & & c^{n-1}\\
\end{pmatrix},
\end{align*}
i.e., 
\begin{align*}
    D(c)^{-1} \theta_{\tau,\alpha} D(c) = \theta_{c^{-n}\tau,c\alpha}.
\end{align*}
Moreover, $D(c_1)D(c_2) = D(c_1 c_2)$. Thus, we obtain a morphism
\begin{align*}
    \bar{s}_{\mathscr{L}}: [V^{\rm sat,\times}_{\mathscr{L}} / \mathbb{G}_m] \rightarrow \mathscr{M}^{\rm cyc,\times}_X,
\end{align*}
which satisfies $\bar{q}_{\mathscr{L}} = {\rm sd}_X \circ \bar{s}_{\mathscr{L}}$.
\end{proof}

We want to define a finite flat family over $X \times [V^{\rm sat,\times}_\mathscr{L} / \mathbb{G}_m]$, which we regard as an analogue of the universal family of spectral covers. Let
\begin{align*}
    p_X: X \times V^{\rm sat,\times}_\mathscr{L} \rightarrow X, \, p_V: X \times V^{\rm sat,\times}_\mathscr{L} \rightarrow V^{\rm sat,\times}_\mathscr{L}
\end{align*}
be projections. There exist universal sections
\begin{align*}
    \tau_{\rm univ} \in H^0(X \times V^{\rm sat,\times}_\mathscr{L}, p^*_X \mathscr{L}^n ), \, \alpha_{\rm univ} \in H^0(X \times V^{\rm sat,\times}_\mathscr{L}, p^*_X(\Omega_X^1 \otimes \mathscr{L}^{-1})).
\end{align*}
In the same way as $\mathcal{A}_\tau$ was defined, we set
\begin{align*}
    \mathcal{A}_{\tau,{\rm univ}} := \oplus_{i=0}^{n-1} p^*_X \mathscr{L}^{-i}
\end{align*}
whose multiplication is defined by $\tau_{\rm univ}$, and define
\begin{align*}
    \mathcal{X}_{\mathscr{L},V} := \Spec \, \mathcal{A}_{\tau, {\rm univ}}
\end{align*}
together with a natural finite flat morphism of degree $n$
\begin{align*}
    p_{\mathscr{L},V} : \mathcal{X}_{\mathscr{L},V} \rightarrow X \times V^{\rm sat,\times}_\mathscr{L}.
\end{align*}
The fiber over any point $(\tau,\alpha) \in V^{\rm sat,\times}_\mathscr{L}$ is exactly $X'_a$, where $a=\tau \alpha^n$. 

Now we pull back the section $\alpha_{\rm univ} \in H^0(X \times V^{\rm sat,\times}_\mathscr{L}, p^*_X(\Omega_X^1 \otimes \mathscr{L}^{-1}))$ to $\mathcal{X}_{\mathscr{L},V}$ and we have
\begin{align*}
    p_{\mathscr{L},V}^* (\alpha_{\rm univ}) : p_{\mathscr{L},V}^* ( p^*_X \mathscr{L} ) \rightarrow p_{\mathscr{L},V}^* ( p^*_X \Omega_X^1 ).
\end{align*}
Applying the tautological section of $p_{\mathscr{L},V}^* ( p^*_X \mathscr{L} )$, we obtain a morphism
\begin{align*}
    \mathcal{O}_{\mathcal{X}_{\mathscr{L},V}} \rightarrow p_{\mathscr{L},V}^* ( p^*_X \Omega_X^1).
\end{align*}
Applying $(p_{\mathscr{L},V})_*$, we obtain a morphism
\begin{align*}
    \theta_{\mathscr{L},V} : \mathcal{E}_{\mathscr{L},V} & := (p_{\mathscr{L},V})_* \mathcal{O}_{\mathcal{X}_{\mathscr{L},V}} \rightarrow (p_{\mathscr{L},V})_* p_{\mathscr{L},V}^*  p^*_X \Omega_X^1 \\
    & \cong (p_{\mathscr{L},V})_* \mathcal{O}_{\mathcal{X}_{\mathscr{L},V}} \otimes p^*_X \Omega_X^1 = \mathcal{E}_{\mathscr{L},V} \otimes p^*_X \Omega_X^1.
\end{align*}
Furthermore, regarding the section $\alpha_{\rm univ}$ as a morphism $p^*_X \mathscr{L} \rightarrow p^*_X \Omega_X^1$, it induces a morphism
\begin{align*}
    \Phi_{\mathscr{L},V}: \mathcal{X}_{\mathscr{L},V} \hookrightarrow {\rm tot}(p^*_X \mathscr{L}) \rightarrow {\rm tot}(p^*_X \Omega_X^1) = T^*_X \times V^{\rm sat,\times}_\mathscr{L}.
\end{align*}

\begin{thm}\label{thm_univ_fam}
There exists a finite flat morphism of degree $n$
\begin{align*}
    p_\mathscr{L} : \mathcal{X}_\mathscr{L} \rightarrow X \times [V^{\rm sat,\times}_\mathscr{L} / \mathbb{G}_m]
\end{align*}
satisfying the following conditions:
\begin{enumerate}
\item The pullback of $p_\mathscr{L}$ along the atlas $X \times V^{\rm sat,\times}_\mathscr{L} \rightarrow X \times [V^{\rm sat,\times}_\mathscr{L} / \mathbb{G}_m]$ is isomorphic to $p_{\mathscr{L},V}$.
\begin{center}
\begin{tikzcd}
\mathcal{X}_{\mathscr{L},V} \arrow[r,"p_{\mathscr{L},V}"] \arrow[d] & X \times V^{\rm sat,\times}_\mathscr{L} \arrow[d] \\
\mathcal{X}_\mathscr{L} \arrow[r,"p_\mathscr{L}"] & X \times {[V^{\rm sat,\times}_\mathscr{L} / \mathbb{G}_m]}
\end{tikzcd}    
\end{center}

\item There exists a morphism
\begin{align*}
\Phi_\mathscr{L}: \mathcal{X}_\mathscr{L} \rightarrow T^*_X \times [V^{\rm sat,\times}_\mathscr{L} / \mathbb{G}_m] 
\end{align*}
such that the diagram
\begin{center}
\begin{tikzcd}
\mathcal{X}_\mathscr{L} \arrow[rr, "\Phi_\mathscr{L}"] \arrow[rd, "p_\mathscr{L}"] & &  T^*_X \times [V^{\rm sat,\times}_\mathscr{L} / \mathbb{G}_m] \arrow[ld] \\
& X \times [V^{\rm sat,\times}_\mathscr{L} / \mathbb{G}_m] &
\end{tikzcd}    
\end{center}
commutes.

\item For every geometric point $(\tau,\alpha)$ of $[V^{\rm sat,\times}_\mathscr{L} / \mathbb{G}_m]$, we have
\begin{align*}
    (\mathcal{X}_\mathscr{L})_{(\tau,\alpha)} \cong X'_a,
\end{align*}
where $a=\tau \alpha^n$.
\end{enumerate}
\end{thm}

\begin{proof}
The $\mathbb{G}_m$-action on the total space ${\rm tot}(p_X^* \mathscr{L})$ is defined as 
\begin{align*}
    \mathbb{G}_m \times {\rm tot}(p_X^* \mathscr{L}) \rightarrow {\rm tot}(p_X^* \mathscr{L}), \quad (c,(\eta,\tau,\alpha)) \mapsto (c^{-1}\eta, c^{-n} \tau, c\alpha),
\end{align*}
which induces a $\mathbb{G}_m$-action on $\mathcal{X}_{\mathscr{L},V}$. We define
\begin{align*}
    \mathcal{X}_\mathscr{L}:=[\mathcal{X}_{\mathscr{L},V} / \mathbb{G}_m].
\end{align*}
Since the morphism 
\begin{align*}
    p_{\mathscr{L},V}: \mathcal{X}_{\mathscr{L},V} \rightarrow X \times V^{\rm sat,\times}_\mathscr{L}
\end{align*}
is $\mathbb{G}_m$-equivariant, the morphism $p_{\mathscr{L},V}$ descends to a finite flat morphism
\begin{align*}
    p_\mathscr{L}: \mathcal{X}_\mathscr{L} =[\mathcal{X}_{\mathscr{L},V} / \mathbb{G}_m] \rightarrow [X \times V^{\rm sat,\times}_\mathscr{L} / \mathbb{G}_m] \cong X \times [ V^{\rm sat,\times}_\mathscr{L} / \mathbb{G}_m],
\end{align*}
which satisfies the first condition by definition. 

We equip the cotangent bundle $T^*_X$ with the trivial $\mathbb{G}_m$-action. Then, we have
\begin{align*}
    [(T^*_X \times V^{\rm sat,\times}_\mathscr{L}) / \mathbb{G}_m] \cong T^*_X \times [V^{\rm sat,\times}_\mathscr{L} / \mathbb{G}_m].
\end{align*}
Therefore, the $\mathbb{G}_m$-equivariant morphism $\Phi_{\mathscr{L},V}$ descends to 
\begin{align*}
    \Phi_\mathscr{L}: \mathcal{X}_\mathscr{L} \rightarrow T^*_X \times [V^{\rm sat,\times}_\mathscr{L} / \mathbb{G}_m],
\end{align*}
which satisfies the second condition.

The third condition is clear from the construction.
\end{proof}

\begin{rem}
Define $E_\mathscr{L}:=(p_\mathscr{L})_* \mathcal{O}_{\mathcal{X}_\mathscr{L}}$ together with a Higgs field 
\begin{align*}
    \theta_\mathscr{L} : \mathcal{E}_\mathscr{L} \rightarrow \mathcal{E}_\mathscr{L} \otimes p^*_X \Omega_X^1
\end{align*}
induced by the Higgs field $\theta_{\mathscr{L},V} : \mathcal{E}_{\mathscr{L},V} \rightarrow \mathcal{E}_{\mathscr{L},V} \otimes p^*_X \Omega_X^1$ on $E_{\mathscr{L},V}$. Then, we obtain a $[V^{\rm sat,\times}_\mathscr{L} /  \mathbb{G}_m]$-family of Higgs bundles on $X$, and it induces a morphism $[V^{\rm sat,\times}_\mathscr{L} /  \mathbb{G}_m] \rightarrow \mathscr{M}^{\rm cyc,\times}_X$, which is exactly the section constructed in Theorem \ref{thm_stack_Hit_sect}.
\end{rem}

\section{Projective Case}

In the previous section, we defined
\begin{align*}
    \bar{q}_{\mathscr{L}}: [V^{\rm sat,\times}_{\mathscr{L}} / \mathbb{G}_m] \rightarrow \mathscr{B}^{\rm cyc,\times}_{X}.
\end{align*}
and constructed a morphism $\bar{s}_{\mathscr{L}}: [V^{\rm sat,\times}_{\mathscr{L}} / \mathbb{G}_m] \rightarrow \mathscr{M}^{\rm cyc,\times}_X$ such that $\bar{q}_{\mathscr{L}} = {\rm sd}_X \circ \bar{s}_{\mathscr{L}}$ in Theorem \ref{thm_stack_Hit_sect}. In this section, let $X$ be a connected smooth projective variety and we construct an analogue of the Hitchin section. We start with the following definition.

\begin{defn}
We define $\mathscr{B}^{\rm cyc}_{\mathscr{L}}$ to be the reduced locally closed subscheme of $\mathscr{B}^{\rm cyc,\times}_{X}$ whose underlying space is $\bar{q}_{\mathscr{L}}([V^{\rm sat,\times}_{\mathscr{L}} / \mathbb{G}_m]  )$.
\end{defn}

Based on the definition of $\mathscr{B}^{\rm cyc}_{\mathscr{L}}$, we have the following commutative diagram
\begin{center}
\begin{tikzcd}
{[V_{\mathscr{L}}^{\mathrm{sat}, \times} / \mathbb{G}_m]} \arrow[r,"\bar{s}_\mathscr{L}"] \arrow[d,"\bar{q}_\mathscr{L}"] & \mathscr{M}^{\rm cyc,\times}_X \arrow[d, "{\rm sd}_X"] \\
\mathscr{B}^{\rm cyc}_{\mathscr{L}} \arrow[r, hook, "\iota_\mathscr{L}"] & \mathscr{B}^{\rm cyc, \times}_X \, .
\end{tikzcd}
\end{center}
Although we do not know whether $\mathscr{B}^{\rm cyc, \times}_X$ is reduced or not, we have
\begin{align*}
    \mathscr{B}^{\rm cyc,\times}_X(k) = \bigcup_{\mathscr{L} \in \mathcal{P}_X} \mathscr{B}^{\rm cyc}_\mathscr{L}(k),
\end{align*}
and $\mathscr{B}^{\rm cyc}_{\mathscr{L}}(k)$ consists of nonzero cyclic spectral data $a$ corresponding to a triple $(\mathscr{L},\tau,\alpha)$ given in Corollary \ref{cor_decomp}. 

Recall the notations given in \S\ref{sect_Hit_sect}
\begin{align*}
U_{\mathscr{L}}: = H^0(X,\mathscr{L}^n), \, W_{\mathscr{L}} := H^0(X, \Omega_X^1 \otimes \mathscr{L}^{-1}), \, V_{\mathscr{L}}:= U_{\mathscr{L}} \times W_{\mathscr{L}}
\end{align*}
with a $\mathbb{G}_m$-action on $V_\mathscr{L}$ defined as
\begin{align*}
    c \cdot (\tau,\alpha) : = (c^{-n} \tau, c \alpha).
\end{align*}

\begin{lem}\label{lem_HE3.11}
Let $X$ be a connected smooth projective variety. Let $a$ be nonzero cyclic spectral data with the associated triple $(\mathscr{L},\tau,\alpha)$. We have either $\dim U_{\mathscr{L}} = 1$ or $\dim W_{\mathscr{L}} = 1$. 
\end{lem}

\begin{proof}
This lemma is a generalization of \cite[Proposition 3.11]{HL24}, which proves the case of rank two, and the proof is also similar to that of \cite[Proposition 3.11]{HL24}. We include the idea here for completeness. Suppose that $\dim U_{\mathscr{L}} \geq 2$, i.e., $h^0(X,\mathscr{L}^n) \geq 2$. Then, $\kappa(X,\mathscr{L}) \geq 1$. Applying \cite[Theorem 2 and Proposition 6]{Rei78} or \cite[Theorem 3.10]{HL24}, any two elements in $W^{\rm sat}_\mathscr{L}$ are proportional. Since $W^{\rm sat}_\mathscr{L}$ is open and non-empty, we have $\dim W_{\mathscr{L}} =1$.
\end{proof}

\begin{lem}\label{lem_B_cyc_L}
Let $X$ be a connected smooth projective variety. For every $\mathscr{L} \in \mathcal{P}_X$, the quotient stack $[V^{\rm sat,\times}_\mathscr{L} /  \mathbb{G}_m]$ is represented by a reduced scheme, and $\bar{q}_\mathscr{L}$ induces an isomorphism
\begin{align*}
    [V^{\rm sat,\times}_\mathscr{L} /  \mathbb{G}_m] \cong \mathscr{B}^{\rm cyc}_\mathscr{L}.
\end{align*}
\end{lem}

\begin{proof}
Based on Lemma \ref{lem_HE3.11}, we prove this lemma in two cases. We first assume that $\dim W_{\mathscr{L}} = 1$. We fix a nonzero element $\alpha_0 \in W_\mathscr{L}$. Since $W_\mathscr{L}$ is of one dimension, any element of $W^{\rm sat}_\mathscr{L}$ is of the form $c \alpha_0$ with some $c \in k^\times$. Therefore, any point of $V^{\rm sat,\times}_\mathscr{L}$ can be written uniquely as a pair $(\tau,c\alpha_0)$ with $\tau \in U_\mathscr{L} \backslash \{0\}$ and $c \in k^\times$. Then, it is clear that every $\mathbb{G}_m$-orbit in $V^{\rm sat,\times}_\mathscr{L}$ contains a unique representative in the form $(\tau,\alpha_0)$. This implies that
\begin{align*}
    [V^{\rm sat,\times}_\mathscr{L} / \mathbb{G}_m] \cong U_\mathscr{L} \backslash \{0\}.
\end{align*}
Under this isomorphism, the morphism $\bar{q}_\mathscr{L}$ is exactly
\begin{align*}
    U_\mathscr{L} \backslash \{0\} \rightarrow \mathscr{B}^{\rm cyc}_\mathscr{L}, \quad \tau \mapsto \tau \alpha^n_0,
\end{align*}
and then 
\begin{align*}
    [V^{\rm sat,\times}_\mathscr{L} /\mathbb{G}_m] \cong U_\mathscr{L} \backslash \{0\} \cong \mathscr{B}^{\rm cyc}_\mathscr{L}.
\end{align*}

Now we suppose that $\dim U_{\mathscr{L}} = 1$. We fix an element $\tau_0 \in U_\mathscr{L} \backslash \{0\}$. Any element in $U_\mathscr{L} \backslash \{0\}$ can be written in the form $c \tau_0$ with $c \in k^{\times}$. Under the $\mathbb{G}_m$-action
\begin{align*}
    c \cdot (\tau,\alpha) : = (c^{-n} \tau, c \alpha),
\end{align*}
we have
\begin{align*}
    [V^{\rm sat,\times}_\mathscr{L} / \mathbb{G}_m] \cong [W^{\rm sat}_\mathscr{L} / \mu_n].
\end{align*}
Since the $\mu_n$-action on $W^{\rm sat}_\mathscr{L}$ is free, we have
\begin{align*}
    [V^{\rm sat,\times}_\mathscr{L} / \mathbb{G}_m] \cong W^{\rm sat}_\mathscr{L} / \mu_n,
\end{align*}
which is a reduced scheme. Note that the natural morphism
\begin{align*}
    W^{\rm sat}_\mathscr{L} \rightarrow H^0(X, S^n \Omega_X^1), \quad \alpha \mapsto \tau_0 \alpha^n
\end{align*}
descends to
\begin{align*}
    W^{\rm sat}_\mathscr{L} / \mu_n \rightarrow H^0(X, S^n \Omega_X^1),
\end{align*}
which is a locally closed immersion and factors through $\mathscr{B}^{\rm cyc,\times}_X$. Therefore, we have
\begin{align*}
    [V^{\rm sat,\times}_\mathscr{L} / \mathbb{G}_m] \cong W^{\rm sat}_\mathscr{L} / \mu_n \cong \mathscr{B}^{\rm cyc}_\mathscr{L}.
\end{align*}
\end{proof}

\begin{thm}\label{thm_proj_Hit_sect}
For each $\mathscr{L} \in \mathcal{P}_X$, there exists a morphism $s_{\mathscr{L}} : \mathscr{B}^{\rm cyc}_{\mathscr{L}} \rightarrow \mathscr{M}^{\rm cyc,\times}_X$ such that ${\rm sd}_X \circ s_{\mathscr{L}} = \iota_\mathscr{L}$.
\begin{center}
\begin{tikzcd}
{[V_{\mathscr{L}}^{\mathrm{sat}, \times} / \mathbb{G}_m]} \arrow[r,"\bar{s}_\mathscr{L}"] \arrow[d,"\cong"] & \mathscr{M}^{\rm cyc,\times}_X \arrow[d, "{\rm sd}_X"] \\
\mathscr{B}^{\rm cyc}_{\mathscr{L}} \arrow[r, hook, "\iota_\mathscr{L}"] \arrow[ru, "s_\mathscr{L}" description, dotted] & \mathscr{B}^{\rm cyc, \times}_X \, .
\end{tikzcd}
\end{center}
\end{thm}

\begin{proof}
This theorem is a direct result of Lemma \ref{lem_B_cyc_L} and Theorem \ref{thm_stack_Hit_sect}. For completeness, we give the proof in two cases as discussed in Lemma \ref{lem_B_cyc_L}. We define the locally free sheaf
\begin{align*}
    E_\mathscr{L} : = \mathcal{O}_X \oplus \mathscr{L}^{-1} \oplus \cdots \oplus \mathscr{L}^{-(n-1)}.
\end{align*}

For the case $\dim W_{\mathscr{L}} = 1$, we fix a nonzero element $\alpha_0 \in W_{\mathscr{L}}$. We define a morphism
\begin{align*}
    U_{\mathscr{L}} \backslash \{0\} \rightarrow \mathscr{M}_X, \quad \tau \mapsto (E_\mathscr{L} , \theta_{\tau,\alpha_0}),
\end{align*}
where 
\begin{align*}
\theta_{\tau,\alpha_0} = 
\begin{pmatrix}
     & \alpha_0 & & \\
    &  & \ddots & \\
    & & & \alpha_0 \\
    \tau \alpha_0 & & & \\
\end{pmatrix}.
\end{align*}
Under the isomorphism $U_{\mathscr{L}} \backslash \{0\} \cong \mathscr{B}^{\rm cyc}_\mathscr{L}$ given in Lemma \ref{lem_B_cyc_L}, we define 
\begin{align*}
    s_\mathscr{L}: \mathscr{B}^{\rm cyc}_\mathscr{L} \xrightarrow{\cong} U_{\mathscr{L}} \backslash \{0\} \cong [V^{\rm sat,\times}_\mathscr{L} /  \mathbb{G}_m] \xrightarrow{\bar{s}_\mathscr{L}} \mathscr{M}^{\rm cyc, \times}_X
\end{align*}
as the composition of morphisms. It is easy to check that the spectral data corresponding to $(E_\mathscr{L},\theta_{\tau,\alpha_0})$ is $a = \tau \alpha_0^n$. Therefore, $s_\mathscr{L} : \mathscr{B}^{\rm cyc}_\mathscr{L} \rightarrow \mathscr{M}^{\rm cyc, \times}_X$ is a section.

For the case $\dim U_{\mathscr{L}} = 1$, we have
\begin{align*}
    [V^{\rm sat,\times}_\mathscr{L} / \mathbb{G}_m] \cong W^{\rm sat}_\mathscr{L} / \mu_n.
\end{align*}
Then, the morphism $\bar{s}_\mathscr{L}: [V^{\rm sat,\times}_\mathscr{L} / \mathbb{G}_m] \rightarrow \mathscr{M}^{\rm cyc, \times}_X$ constructed in Theorem \ref{thm_stack_Hit_sect} induces a morphism
\begin{align*}
    W^{\rm sat}_\mathscr{L} / \mu_n \rightarrow \mathscr{M}^{\rm cyc, \times}_X.
\end{align*}
Under the isomorphism $\mathscr{B}^{\rm cyc}_\mathscr{L} \cong W^{\rm sat}_\mathscr{L} / \mu_n$, the section is defined as
\begin{align*}
    s_\mathscr{L}: \mathscr{B}^{\rm cyc}_\mathscr{L} \xrightarrow{\cong} W^{\rm sat}_\mathscr{L} / \mu_n \xrightarrow{\cong} [V^{\rm sat,\times}_\mathscr{L} / \mathbb{G}_m] \xrightarrow{\bar{s}_\mathscr{L}} \mathscr{M}^{\rm cyc, \times}_X.
\end{align*}
\end{proof}

\begin{rem}\label{rem_non-alg}
Although we construct a section $s_\mathscr{L}: \mathscr{B}^{\rm cyc}_\mathscr{L} \rightarrow \mathscr{M}^{\rm cyc,\times}_X$, it is difficult to construct a ``global" one on the scheme of cyclic spectral data $\mathscr{B}^{\rm cyc}_X$. The main difficulty arises from the decomposition given in Corollary \ref{cor_decomp}, where the uniqueness depends on the condition that $\alpha(\mathscr{L})$ is saturated. We illustrate this difficulty with the following example.

Let $X = C \times E$ be a product of smooth projective curves with projections
\begin{align*}
    p_C: X \rightarrow C, \, p_E: X \rightarrow E.
\end{align*}
Denote by $K_\bullet$ the canonical bundle, where $\bullet = C, E$. Assume that the genus satisfies $g(C) > 1$ and $g(E) = 1$. We write
\begin{align*}
\Omega_X^1 = \mathscr{L}_1 \oplus \mathscr{L}_2,
\end{align*}
where $\mathscr{L}_1:= p_C^* K_C$ and $\mathscr{L}_2 := p_E^* K_E \cong \mathcal{O}_X$. We choose nonzero sections $u_C \in H^0(C,K_C)$ and $u_E \in H^0(E,K_E)$. Now we define
\begin{align*}
    \omega_t:= p^*_C u_C + t p^*_E u_E \in H^0(X,\Omega_X^1)
\end{align*}
for $t \in \mathbb{A}^1$, and set $a_t : = \omega_t^2 \in H^0(X,S^2\Omega_X^1)$. Then, we obtain a morphism
\begin{align*}
    \mathbb{A}^1 \rightarrow \mathscr{B}^{\rm cyc}_X \backslash \{0\}, \quad t \mapsto a_t.
\end{align*}

When $t \neq 0$, $\omega_t$ is nowhere vanishing because $E$ is elliptic and $u_E$ is nowhere vanishing. Then, the morphism $\omega_t:\mathcal{O}_X \cong \mathscr{L}_2 \rightarrow \Omega_X^1$ is saturated and $a_t \in \mathscr{B}^{\rm cyc}_{\mathscr{L}_2}(k)$. Thus, the cyclic Higgs bundle constructed in the second proof of Theorem \ref{thm_nonempty} is $(E_t,\theta_t)$, where
\begin{align*}
    E_t = \mathcal{O}_X \oplus \mathcal{O}_X, \, \theta_t = \begin{pmatrix}
        0 & \omega_t \\
        \omega_t & 0
    \end{pmatrix}.
\end{align*}
When $t=0$, we have $\omega_0 = p^*_C u_C$. Since $g(C) > 1$, the section $u_C$ has a nontrivial zero locus. Then, the morphism $\omega_0: \mathcal{O}_X \rightarrow \Omega_X^1$ is not saturated. Note that $\omega_0$ factors as
\begin{align*}
    \mathcal{O}_X \xrightarrow{\cdot u_C} \mathscr{L}_1 \hookrightarrow \Omega_X^1.
\end{align*}
Clearly, the natural inclusion $\mathscr{L}_1 \hookrightarrow \Omega_X^1$ is saturated and $a_0 \in \mathscr{B}^{\rm cyc}_{\mathscr{L}_1}(k)$. Therefore, the underlying locally free sheaf of the corresponding Higgs bundle is 
\begin{align*}
    E_0 = \mathcal{O}_X \oplus \mathscr{L}_1^{-1}.
\end{align*}

Since
\begin{align*}
    c_1(E_t) = 0 \text{ and } c_1(E_0) = -c_1(\mathscr{L}_1)\neq 0,
\end{align*}
these objects $(E_t,\theta_t)$ for $t \in \mathbb{A}^1$ cannot be fibers of an algebraic flat family over $\mathbb{A}^1$. Therefore, the set-theoretic map
\begin{align*}
    \mathscr{B}^{\rm cyc,\times}_X(k) \rightarrow \mathscr{M}^{\rm cyc, \times}_X(k), \quad a \mapsto (E_a,\theta_a)
\end{align*}
given in the second proof of Theorem \ref{thm_nonempty} is not induced by a morphism $\mathscr{B}^{\rm cyc,\times}_X \rightarrow \mathscr{M}^{\rm cyc, \times}_X$.
\end{rem}

As a direct result of Theorem \ref{thm_univ_fam} and Lemma \ref{lem_B_cyc_L}, there exists a tautological family of spectral covers on $\mathscr{B}^{\rm cyc}_\mathscr{L}$.

\begin{prop}\label{prop_univ_fam_proj}
Let $X$ be a connected smooth projective variety. For each $\mathscr{L} \in \mathcal{P}_X$, there exists a finite flat morphism of degree $n$
\begin{align*}
    \mathcal{X}_\mathscr{L} \rightarrow X \times \mathscr{B}^{\rm cyc}_\mathscr{L}
\end{align*}
such that for every geometric point $a$ of $\mathscr{B}^{\rm cyc}_\mathscr{L}$, the fiber is $(\mathcal{X}_\mathscr{L})_{a} \cong X'_a$.
\end{prop}

\bibliographystyle{amsalpha} 
\bibliography{ref_pnahc}

\bigskip
\noindent\small{\textsc{Department of Mathematics, South China University of Technology, \\
Guangzhou, 510641, China}\\
\emph{E-mail address}:  \texttt{hsun71275@scut.edu.cn}

\end{document}